\documentclass[11pt,reqno]{amsart}

\usepackage[left=3.5cm,right=3.5cm,top=3.5cm,bottom=3.5cm]{geometry}
\usepackage{tipa}
\usepackage{amssymb}
\usepackage{graphicx}
\usepackage[hidelinks]{hyperref}
\usepackage{enumerate}
\usepackage{xcolor}
\usepackage{amsmath}
\allowdisplaybreaks[4]

\numberwithin{equation}{section}
\newtheorem{theorem}{Theorem}[section]

\newtheorem{corollary}{Corollary}[section]
\newtheorem{lemma}[theorem]{Lemma}

\theoremstyle{definition}

\newtheorem*{remarks*}{Remarks}

\numberwithin{equation}{section}

\title{A congruence involving the quotients of Euler and its applications (III)}

\author[H. Zhong]{Hao Zhong}
\address{(H. Zhong) School of Mathematical Sciences, Zhejiang University, Hangzhou, 310027, China}
\curraddr{}
\email{11435011@zju.edu.cn}
\thanks{}

\author[S. Chern]{Shane Chern}
\address{(S. Chern) School of Mathematical Sciences, Zhejiang University, Hangzhou, 310027, China}
\curraddr{}
\email{chenxiaohang92@gmail.com}
\thanks{}

\author[T. Cai]{Tianxin Cai}
\address{(T. Cai) School of Mathematical Sciences, Zhejiang University, Hangzhou, 310027, China}
\curraddr{}
\email{txcai@zju.edu.cn}
\thanks{}

\keywords{Binomial coefficient, Morley's congruence}
\subjclass[2010]{11A25, 11B65, 11B68}

\begin{document}

\maketitle

\thispagestyle{empty}

\begin{abstract}
In this paper, we will present several new congruences involving binomial coefficients under integer moduli, which are the continuation of the previous two work by Cai \textit{et al.} (2002, 2007).
\end{abstract}

\section{Introduction}

In 1895, Morley \cite{Mor1895} proved the following beautiful and profound congruence involving binomial coefficients, that is, for any prime $p\ge 5$,
\begin{equation}\label{eq:Morley}
(-1)^{(p-1)/2}\binom{p-1}{(p-1)/2}\equiv 4^{p-1} \pmod{p^3}.
\end{equation}
However, his proof, which is due to an explicit form of De Moivre's theorem, fails to deal with other binomial coefficients. In 2002, Cai \cite{Cai2002} extended Morley's congruence to integer moduli through a generalization of Lehmer's congruence. More precisely, he proved
\begin{equation}\label{eq:Cai1}
\prod_{d\mid n}\binom{d-1}{(d-1)/2}^{\mu(n/d)}\equiv (-1)^{\phi(n)/2}4^{\phi(n)}
\begin{cases}
(\bmod{\ n^3}) & \text{if }3\nmid n,\\
(\bmod{\ n^3/3}) & \text{if }3\mid n,
\end{cases}
\end{equation}
for odd $n> 1$. When $n$ is an odd prime $p\ge 5$, \eqref{eq:Cai1} becomes \eqref{eq:Morley}. According to Cosgrave and Dilcher \cite{CD2013}, \eqref{eq:Cai1} appears to be the first analogue for composite moduli of a ``Lehmer type'' congruence. Later in 2007, Cai \textit{et al.} \cite{CFZ2007} proposed several new congruences of the same type, in which $(d-1)/2$ is replaced by $\lfloor d/3\rfloor$, $\lfloor d/4\rfloor$, and $\lfloor d/6\rfloor$, respectively, where $\lfloor x\rfloor$ denotes the largest integer not greater than $x$. In this paper, we will further extend the work in \cite{Cai2002} and \cite{CFZ2007}.

First, we introduce a new generalization of Euler's totient. For $k$ an integer and $f$ a number theoretic function, we define
\begin{equation}\label{eq:Eulertot1}
\phi^{(k)}_{f}(n):=\sum_{d\mid n}(\frac{n}{d})^{k}f(d)\mu(d).
\end{equation}
And if $f\equiv1$, then $\phi^{(k)}_{f}(n)$ equals to Jordan totient function. It's easy to prove that $\phi^{(k)}_{f}(n)=n^{k}\prod_{p\mid n}(1-f(p)p^{-k})$ when $f(n)$ is multiplicative. And Our main results are

\begin{theorem}\label{th:1}
Let $n$ be a positive integer and $(n,6)=1$. For $e=2$, $3$, $4$ or $6$, we have

\begin{equation}\label{eq:th:1}
\sum^{\lfloor n/e\rfloor}_{\substack{r=1\\ (r,n)=1}}\frac{1}{r^{2}}\equiv -J_{e}(n)n^{\phi(n)-2}\phi_{J_{e}}^{(2-\phi(n))}(n)\frac{B_{\phi(n)-1}(\frac{1}{e})}{\phi(n)-1}\pmod{n},
\end{equation}
where $B_{n}$ is the $n$th Bernoulli number, $B_{n}(x)$ is the Bernoulli polynomial, and $J_{e}(n)$ is the Jacobi symbol for $n$ and $e$,
\begin{equation*}
J_{e}(n)=(\frac{n}{e})=
\begin{cases}
1  & \text{if }n\equiv 1 \pmod e\\
-1 & \text{if }n\equiv -1 \pmod e
\end{cases}
\end{equation*}
since $(n,6)=1$.
\end{theorem}

According to Lehmer \cite{Lehmer1938}, for $v$ odd, $B_{v}(\frac{1}{2})=0$ and $B_{v}(\frac{1}{4})=-\frac{vE_{v-1}}{4^{v}}$ where $E_{m}$ is the $m$th Euler number defined by the generating function
\begin{equation*}
\frac{1}{\cosh x}=\sum_{m=0}^{}\frac{E_{m}}{m!}\cdot x^{m}.
\end{equation*}
It follows
\begin{corollary}\label{cor:1}
For $n$ integer and (n,6)=1,
\begin{equation}\label{eq:cor:1}
\sum^{\frac{n-1}{2}}_{\substack{r=1\\ (r,n)=1}}\frac{1}{r^{2}}\equiv 0 \pmod n.
\end{equation}
\end{corollary}

\begin{corollary}\label{cor:2}
For $n$ integer and (n,6)=1,
\begin{equation}\label{eq:cor:2}
\sum^{\lfloor n/4\rfloor}_{\substack{r=1\\ (r,n)=1}}\frac{1}{r^{2}}\equiv(-1)^{\frac{n-1}{2}}4n^{\phi(n)-2}\phi_{J_{4}}^{(2-\phi(n))}(n)E_{\phi(n)-2}\pmod{n}
\end{equation}
\end{corollary}

\begin{corollary}\label{cor:3}
For $n$ a positive integer, we have
\begin{enumerate}[{\rm (1)}]
\item if $(n,6)=1$,
\begin{equation}\label{eq:cor:3.1}
\sum_{\substack{r=1\\ (r,n)=1}}^{\lfloor n/3\rfloor}\frac{1}{r}\equiv -\frac{3}{2}q_3(n)+\frac{3}{4}nq_3^2(n)+\frac{1}{3}J_{3}(n)n^{\phi(n)-1}\phi_{J_{3}}^{(2-\phi(n))}(n)\frac{B_{\phi(n)-1}(\frac{1}{3})}{\phi(n)-1} \pmod{n^2};
\end{equation}
\item if $(n,6)=1$,
\begin{equation}\label{eq:cor:3.2}
\sum_{\substack{r=1\\ (r,n)=1}}^{\lfloor n/4\rfloor}\frac{1}{r}\equiv -3q_2(n)+\frac{3}{2}nq_2^2(n)+(-1)^{\frac{n+1}{2}}n^{\phi(n)-1}\phi_{J_{4}}^{(2-\phi(n))}(n)E_{\phi(n)-2} \pmod{n^2};
\end{equation}
\item if $(n,30)=1$,
\newpage
\begin{align}\label{eq:cor:3.3}
\sum_{\substack{r=1\\ (r,n)=1}}^{\lfloor n/6\rfloor}\frac{1}{r}\equiv -2q_2(n)-\frac{3}{2}q_3(n)+nq_2^2(n)+\frac{3}{4}nq_3^2(n)+\frac{1}{6}J_{6}(n)n^{\phi(n)-1}\phi_{J_{6}}^{(2-\phi(n))}(n)\frac{B_{\phi(n)-1}(\frac{1}{6})}{\phi(n)-1}\notag\\ \pmod{n^2}.
\end{align}
\end{enumerate}
\end{corollary}

\begin{theorem}\label{th:2}
For any positive integer $k$ and odd $n>1$, it follows
\begin{equation}\label{eq:th:2}
\prod_{d\mid n}\binom{kd-1}{(d-1)/2}^{\mu(n/d)}\equiv (-1)^{\phi(n)/2}4^{k\phi(n)}
\begin{cases}
(\bmod{\ n^3}) & \text{if }3\nmid n,\\
(\bmod{\ n^3/3}) & \text{if }3\mid n.
\end{cases}
\end{equation}
\end{theorem}

\begin{remarks*}
(1). One immediately sees that \eqref{eq:th:2} gives a generalization of Morley's congruence \eqref{eq:Morley} if we let $n$ be an odd prime $p\ge 5$.
\begin{corollary}\label{cor:4}
Let $p\ge 5$ be an odd prime. For any positive integer $k$, it follows
\begin{equation}\label{eq:cor:4}
(-1)^{(p-1)/2}\binom{kp-1}{(p-1)/2}\equiv 4^{k(p-1)} \pmod{p^3}.
\end{equation}
\end{corollary}
\noindent (2). If we let $n$ be the product of two distinct odd primes, we obtain the following result resembling the quadratic reciprocity law, which extends Cai's Corollary 4 in \cite{Cai2002}.
\begin{corollary}\label{cor:5}
Let $p$, $q$ be two distinct odd primes. For any positive integer $k$, it follows
\begin{equation}\label{eq:cor:5}
\binom{kpq-1}{(pq-1)/2}\equiv 4^{k(p-1)(q-1)}\binom{kp-1}{(p-1)/2}\binom{kq-1}{(q-1)/2} \pmod{p^3q^3}.
\end{equation}
\end{corollary}
\end{remarks*}

In order to make the the following theorem briefly, we set
\begin{equation*}
A_{e}(n):=J_{e}(n)n^{\phi(n)-2}\phi^{(2-\phi(n))}_{J_{e}}(n)\frac{B_{\phi(n)-1}(\frac{1}{e})}{\phi(n)-1}.
\end{equation*}
Then, we have
\begin{theorem}\label{th:3}
For any positive integer $k$ and odd $n>1$, it follows
\begin{enumerate}[{\rm (1)}]
\item if $(3,n)=1$,
\begin{equation}\label{eq:th:3.1}
\prod_{d\mid n}\binom{kd-1}{\lfloor d/3\rfloor}^{\mu(n/d)}\equiv (-1)^{\phi_3(n)}\left\{\frac{1}{2}(27^{k\phi(n)}+1)+k(\frac{1}{2}k-\frac{1}{3})n^{2}A_{3}(n)\right\} \pmod{n^3};
\end{equation}
\item if $(3,n)=1$,
\begin{align}\label{eq:th:3.2}
\prod_{d\mid n}\binom{kd-1}{\lfloor d/4\rfloor}^{\mu(n/d)}\equiv (-1)^{\phi_4(n)}\left\{8^{k\phi(n)}+(-1)^{\frac{n+1}{2}}k(2k-1)n^{\phi(n)}\phi_{J_{4}}^{(2-\phi(n))}(n)E_{\phi(n)-2}\right\} \notag\\
\pmod{n^3};
\end{align}
\item if $(15,n)=1$,
\begin{equation}\label{eq:th:3.3}
\prod_{d\mid n}\binom{kd-1}{\lfloor d/6\rfloor}^{\mu(n/d)}\equiv (-1)^{\phi_6(n)}\left\{\frac{1}{2}(16^{k\phi(n)}+27^{k\phi(n)})+\frac{1}{2}k(k-\frac{1}{3})n^{2}A_{6}(n)\right\} \pmod{n^3}.
\end{equation}
\end{enumerate}
Here
$$\phi_e(n):=\sum_{d\mid n}\mu\left(\frac{n}{d}\right)\left\lfloor\frac{d}{e}\right\rfloor$$
is the generalized Euler totient function defined in \cite{CFZ2007}.
\end{theorem}

Furthermore, we consider the following generalized binomial coefficient. For $x\in\mathbb{C}$, let
$$\binom{x}{n}=\frac{x(x-1)\cdots(x-n+1)}{n(n-1)\cdots 1},$$
if $n$ is a positive integer, and $\binom{x}{0}=1$. If we replace $kd-1$ in Theorem \ref{th:2} by $(kd-1)/2$, then we have a new congruence.

\begin{theorem}\label{th:4}
For positive integer $k$ and odd $n>1$, it follows
\begin{equation}\label{eq:th:4}
\prod_{d\mid n}\binom{(kd-1)/2}{(d-1)/2}^{\mu(n/d)}\equiv 2^{-(k-1)\phi(n)}
\begin{cases}
(\bmod{\ n^3}) & \text{if }3\nmid n,\\
(\bmod{\ n^3/3}) & \text{if }3\mid n.
\end{cases}
\end{equation}
\end{theorem}

\section{Preliminaries}

To prove the theorems, we need the following lemmas.

\begin{lemma}[{\cite[Lemma 1]{CFZ2007}}]\label{lem:1}
If $p\geq5$ is a prime and $k\geq2$, $l$, $t$ are positive integers, and $s$ is the least positive residue of $p^{l}$ modulo $t$, then
\begin{equation}\label{eq:lem:1}
\sum_{r=1}^{\lfloor p^{l}/t\rfloor}(p^{l}-tr)^{2k}\equiv \frac{t^{2k}}{2k+1}\{\frac{2k+1}{t}p^{l}B_{2k}-B_{2k+1}(\frac{s}{t})\}\pmod {p^{3l-1}},
\end{equation}
where $B_{n}$ is the $n$th Bernoulli number.
\end{lemma}

\begin{lemma}[{\cite[Lemma 1]{Cai2002}}]\label{lem:2}
Let $n>1$ be an integer. Then
\begin{equation}\label{eq:lem:2}
\sum_{\substack{i=1\\ (i,n)=1}}^{n-1}\frac{1}{i^2}\equiv 0
\begin{cases}
(\bmod{\ n}) & \text{if }3\nmid n,\ n\ne 2^a,\\
(\bmod{\ n/3}) & \text{if }3\mid n,\\
(\bmod{\ n/2}) & \text{if }n=2^a.
\end{cases}
\end{equation}
\end{lemma}

\begin{lemma}[{\cite[Corollary 1.3]{Sun2003}}]\label{lem:3}
Let $a\in\mathbb{Z}$, $k$, $q$, $m\in\mathbb{Z}^{+}$ and $(m,q)=1$. Then
\begin{equation}\label{eq:lem:3}
\frac{1}{k}(m^{k}B_{k}(\frac{x+a}{m})-B_{k}(x))\equiv \sum_{j=0}^{q-1}(\lfloor\frac{a+jm}{q}\rfloor+\frac{1-m}{2})(x+a+jm)^{k-1}\pmod q.
\end{equation}
\end{lemma}

\begin{lemma}[{\cite[Theorem 1]{Cai2002}}]\label{lem:4}
For odd $n>1$, we have
\begin{equation}\label{eq:lem:4}
\sum_{\substack{i=1\\ (i,n)=1}}^{(n-1)/2} \frac{1}{i} \equiv -2q_2(n)+nq_2^2(n) \pmod{n^2},
\end{equation}
where $q_r(n)$ denotes the Euler quotient, i.e.,
$$q_r(n)=\frac{r^{\phi(n)}-1}{n},$$
with $(n,r)=1$.
\end{lemma}

\begin{lemma}[{\cite[Theorem 1]{CFZ2007}}]\label{lem:5}
For odd $n>1$, we have
\begin{enumerate}[{\rm (1)}]
\item if $(3,n)=1$,
\begin{equation}\label{eq:lem:5.1}
\sum_{\substack{r=1\\ (r,n)=1}}^{\lfloor d/3\rfloor}\frac{1}{n-3r}\equiv \frac{1}{2}q_3(n)-\frac{1}{4}nq_3^2(n) \pmod{n^2};
\end{equation}
\item if $(3,n)=1$,
\begin{equation}\label{eq:lem:5.2}
\sum_{\substack{r=1\\ (r,n)=1}}^{\lfloor d/4\rfloor}\frac{1}{n-4r}\equiv \frac{3}{4}q_2(n)-\frac{3}{8}nq_2^2(n) \pmod{n^2};
\end{equation}
\item if $(15,n)=1$,
\begin{equation}\label{eq:lem:5.3}
\sum_{\substack{r=1\\ (r,n)=1}}^{\lfloor d/6\rfloor}\frac{1}{n-6r}\equiv \frac{1}{3}q_2(n)+\frac{1}{4}q_3(n)-\frac{1}{6}nq_2^2(n)-\frac{1}{8}nq_3^2(n) \pmod{n^2}.
\end{equation}
\end{enumerate}
\end{lemma}

\section{Proofs of the Theorems}

\begin{proof}[Proof of Theorem \ref{th:1}]
First of all, we prove that for $p$ prime and $l$ positive integer,
\begin{equation}\label{eq:1.0}
\sum^{\lfloor p^{l}/e\rfloor}_{\substack{r=1\\ p\nmid r}}\frac{1}{r^{2}}\equiv -J_{e}(p^{l})\frac{B_{\phi(p^{l})-1}(\frac{1}{e})}{\phi(p^{l})-1}\pmod{p^{l}}.
\end{equation}
By using the multiplication theorems given by Joseph Ludwig Raabe in 1851: for natural number $m\geq1$, $B_{n}(mx)=m^{n-1}\sum_{k=0}^{m-1}B_{n}(x+\frac{k}{m})$, we can obtain the values of $B_{n}(\frac{s}{t})$. Fix $x$ zero and $n$ odd,\\

if $m=2$, then $B_{n}(\frac{1}{2})=0$;

if $m=3$, then $B_{n}(\frac{1}{3})=-B_{n}(\frac{2}{3})$;

if $m=4$, then $B_{n}(\frac{1}{4})=-B_{n}(\frac{3}{4})$;

if $m=6$, then $B_{n}(\frac{1}{6})=-B_{n}(\frac{5}{6})$.

Taking $2k=\phi(p^{l})-2$ and $t=e$ in \eqref{eq:lem:1} and using the von Staudt-Clauson theorem, we have
\begin{align*}
\sum_{\substack{r=1\\ p\nmid r}}^{\lfloor p^{l}/e\rfloor}\frac{1}{r^{2}} &\equiv\sum_{\substack{r=1\\ p\nmid r}}^{\lfloor p^{l}/e\rfloor}\frac{e^{2}}{(p^{l}-er)^{2}} \equiv\sum_{r=1}^{\lfloor p^{l}/e\rfloor}(p^{l}-er)^{\phi(p^{l})-2}\\ &\equiv -\frac{e^{\phi(p^{l})-2+2}}{\phi(p^{l})-1}B_{\phi(p^{l})-1}(\frac{s}{e})\\ &\equiv-\frac{B_{\phi(p^{l})-1}(\frac{s}{e})}{\phi(p^{l})-1} \equiv-J_{e}(p^l)\frac{B_{\phi(p^{l})-1}(\frac{1}{e})}{\phi(p^{l})-1}\pmod {p^{l}},
\end{align*}
where $s$ is the least positive residue of $p^{l}$ modulo $e$. Hence, \eqref{eq:1.0} is valid.

Secondly, we prove that for a positive integer $m$, if $(m,e)=1$, then
\begin{equation}\label{eq:1.1}
\sum_{\substack{r=1\\ p\nmid r}}^{\lfloor mp^{l}/e\rfloor}\frac{1}{r^{2}}\equiv J_{e}(m)\sum_{\substack{r=1\\ p\nmid r}}^{\lfloor p^{l}/e\rfloor}\frac{1}{r^{2}}.
\end{equation}
Using Lemma \ref{lem:2}, we have

if $m\equiv1\pmod e$, then $m=ek+1$ for some nonnegative integer $k$ and
\begin{align*}
\sum_{\substack{r=1\\ p\nmid r}}^{\lfloor (ek+1)p^{l}/e\rfloor}\frac{1}{r^{2}} &= \sum_{\substack{r=1\\ p\nmid r}}^{\lfloor kp^{l}+p^{l}/e\rfloor}\frac{1}{r^{2}} = \sum_{\substack{r=1\\ p\nmid r}}^{kp^{l}}\frac{1}{r^{2}}+\sum_{\substack{r=kp^{l}+1\\ p\nmid r}}^{kp^{l}+\lfloor p^{l}/e\rfloor}\frac{1}{r^{2}}\\
&\equiv \sum_{a=0}^{k-1}\sum_{\substack{b=1\\ p\nmid b}}^{p^{l}}\frac{1}{(ap^{l}+b)^{2}}+\sum_{\substack{r=1\\ p\nmid r}}^{\lfloor p^{l}/e\rfloor}\frac{1}{(kp^{l}+r)^{2}}\\
&\equiv \sum_{a=0}^{k-1}\sum_{\substack{b=1\\ p\nmid b}}^{p^{l}}\frac{1}{b^{2}}+\sum_{\substack{r=1\\ p\nmid r}}^{\lfloor p^{l}/e\rfloor}\frac{1}{r^{2}}\\
&\equiv k\sum_{\substack{b=1\\ p\nmid b}}^{p^{l}-1}\frac{1}{b^{2}}+\sum_{\substack{r=1\\ p\nmid r}}^{\lfloor p^{l}/e\rfloor}\frac{1}{r^{2}}\\
&\equiv \sum_{\substack{r=1\\ p\nmid r}}^{\lfloor p^{l}/e\rfloor}\frac{1}{r^{2}}\pmod {p^{l}};
\end{align*}

if $m\equiv-1\pmod e$, then $m=ek-1$ for some positive integer $k$ and
\begin{align*}
\sum_{\substack{r=1\\ p\nmid r}}^{\lfloor (ek-1)p^{l}/e\rfloor}\frac{1}{r^{2}} &= \sum_{\substack{r=1\\ p\nmid r}}^{\lfloor (k-1)p^{l}+(e-1)p^{l}/e\rfloor}\frac{1}{r^{2}} \equiv \sum_{\substack{r=1\\ p\nmid r}}^{\lfloor (e-1)p^{l}/e\rfloor}\frac{1}{r^{2}}\\
&\equiv \sum_{\substack{r=1\\ p\nmid r}}^{p^{l}-\lfloor p^{l}/e\rfloor-1}\frac{1}{r^{2}} \equiv \sum_{\substack{r=1\\ p\nmid r}}^{p^{l}-1}\frac{1}{r^{2}}-\sum_{\substack{r=1\\ p\nmid r}}^{\lfloor p^{l}/e\rfloor}\frac{1}{r^{2}}\\
&\equiv -\sum_{\substack{r=1\\ p\nmid r}}^{\lfloor p^{l}/e\rfloor}\frac{1}{r^{2}}\pmod {p^{l}}.
\end{align*}
So \eqref{eq:1.1} is valid. Furthermore, if $p^{l}||n$, taking $m=\frac{n}{p^{l}}$ into \eqref{eq:1.1} yields that
\begin{equation}\label{eq:1.2}
\sum^{\lfloor n/e\rfloor}_{\substack{r=1\\ p\nmid r}}\frac{1}{r^{2}}\equiv -J_{e}(n)\frac{B_{\phi(p^{l})-1}(\frac{1}{e})}{\phi(p^{l})-1}\pmod{p^{l}}.
\end{equation}

Thirdly, we prove that
\begin{equation}\label{eq:1.3}
\sum^{\lfloor n/e\rfloor}_{\substack{r=1\\ (r,n)=1}}\frac{1}{r^{2}}\equiv -J_{e}(n)n^{\phi(n)-2}\phi_{J_{e}}^{(2-\phi(n))}(n)\frac{B_{\phi(p^{l})-1}(\frac{1}{e})}{\phi(p^{l})-1}\pmod{p^{l}},
\end{equation}

Assume $p_{1}$, $p_{2}$, \ldots, $p_{u}$ are different prime factors of $n$. By noticing that $\phi(n)-1\geq\phi(p^{l})-1=p^{l-1}(p-1)-1\geq4\cdot5^{l-1}-1>l$, we have
\begin{align*}
\sum_{\substack{r=1\\ (r,n)=1}}^{\lfloor n/e\rfloor}\frac{1}{r^2}
&= \sum_{\substack{r=1\\ p\nmid r}}^{\lfloor n/e\rfloor}\frac{1}{r^2}-\sum_{i}\sum_{\substack{r=1\\ p\nmid r\\ p_{i}\mid r}}^{\lfloor n/e\rfloor}\frac{1}{r^2}+\sum_{i,j}\sum_{\substack{r=1\\ p\nmid r\\ p_{i}p_{j}\mid r}}^{\lfloor n/e\rfloor}\frac{1}{r^2}+\cdots+(-1)^{u}\sum_{\substack{r=1\\ p\nmid r\\ p_{1}p_{2}\cdots p_{u}\mid r}}^{\lfloor n/e\rfloor}\frac{1}{r^2}\\
&\equiv -\left\{J_{e}(n)-\sum_{i}\frac{J_{e}(n/p_{i})}{p_{i}^{2}}+\sum_{i,j}\frac{J_{e}(n/p_{i}p_{j})}{p_{i}^{2}p_{j}^{2}}+\cdots\right.\\
&\quad\quad\quad\quad\quad\quad\left.+(-1)^{u}\frac{J_{e}(n/p_{i}p_{j}\cdots p{u})}{p_{i}^{2}p_{j}^{2}\cdots p_{u}^{2}}\right\}\frac{B_{\phi(p^{l})-1}(\frac{1}{e})}{\phi(p^{l})-1}\\
&\equiv -\left\{J_{e}(n)-\sum_{i}\frac{J_{e}(n)}{p_{i}^{2}J_{e}(p_{i})}+\sum_{i,j}\frac{J_{e}(n)}{p_{i}^{2}p_{j}^{2}J_{e}(p_{i}p_{j})}+\cdots\right.\\
&\quad\quad\quad\quad\quad\quad\left.+(-1)^{u}\frac{J_{e}(n)}{p_{i}^{2}p_{j}^{2}\cdots p_{u}^{2}J_{e}(p_{i}p_{j}\cdots p{u})}\right\}\frac{B_{\phi(p^{l})-1}(\frac{1}{e})}{\phi(p^{l})-1}\\
&\equiv -J_{e}(n)\prod_{q\mid \frac{n}{p^{l}}}(1-\frac{1}{q^{2}J_{e}(q)})\frac{B_{\phi(p^{l})-1}(\frac{1}{e})}{\phi(p^{l})-1}\\
&\equiv -J_{e}(n)\prod_{q\mid n}(1-q^{\phi(p^{l})-2}J_{e}(q))\frac{B_{\phi(p^{l})-1}(\frac{1}{e})}{\phi(p^{l})-1}\\
&\equiv -J_{e}(n)n^{\phi(n)-2}\phi_{J_{e}}^{(2-\phi(n))}(n)\frac{B_{\phi(p^{l})-1}(\frac{1}{e})}{\phi(p^{l})-1}\pmod{p^{l}},
\end{align*}
which means \eqref{eq:1.3} is valid.

Finally, taking $k=\phi(p^l)$, $m=e$, $x=0$, $a=1$ and $q=p^l$ in \eqref{eq:lem:3}, we have
\begin{align*}
\frac{B_{\phi(p^{l})-1}(\frac{1}{e})}{\phi(p^{l})-1} &\equiv e\sum_{j=0}^{p^{l}-1}(\lfloor\frac{1+je}{p^{l}}\rfloor+\frac{1-e}{2})(1+je)^{\phi(p^{l})-2}\pmod {p^{l}}\\
&\equiv e\sum_{\substack{j=0\\ (p,1+je)=1}}^{p^{l}-1}(\lfloor\frac{1+je}{p^{l}}\rfloor+\frac{1-e}{2})(1+je)^{-2}\pmod {p^{l}}.
\end{align*}
Changing $k$ to $\phi(n)-1$, we have
\begin{align*}
\frac{B_{\phi(n)-1}(\frac{1}{e})}{\phi(n)-1} &\equiv e\sum_{j=0}^{p^{l}-1}(\lfloor\frac{1+j}{p^{l}}\rfloor+\frac{1-e}{2})(1+je)^{\phi(n)-2}\pmod {p^{l}}\\
&\equiv e\sum_{\substack{j=0\\ (p,1+je)=1}}^{p^{l}-1}(\lfloor\frac{1+je}{p^{l}}\rfloor+\frac{1-e}{2})(1+je)^{-2}\pmod {p^{l}},
\end{align*}
which means for $p^l||n$,
\begin{equation}\label{eq:1.4}
\frac{B_{\phi(n)-1}(\frac{1}{e})}{\phi(n)-1} \equiv \frac{B_{\phi(p^{l})-1}(\frac{1}{e})}{\phi(p^{l})-1} \pmod {p^l}.
\end{equation}
Hence, we can obtain \eqref{eq:th:1}
\end{proof}

\begin{proof}[Proof of Corollary \ref{cor:3}]
It follows by Theorem \ref{th:1} and Lemma \ref{lem:5}. Since
\begin{align*}
\sum^{\lfloor n/e\rfloor}_{\substack{r=1\\ (r,n)=1}}\frac{1}{n-er} &\equiv \sum^{\lfloor n/e\rfloor}_{\substack{r=1\\ (r,n)=1}}(n-er)^{\phi(n^2)-1}\\
&\equiv \sum^{\lfloor n/e\rfloor}_{\substack{r=1\\ (r,n)=1}}\left\{(-er)^{\phi(n^2)-1}+(\phi(n^2)-1)n(-er)^{\phi(n^2)-2}\right\}\\
&\equiv -\sum^{\lfloor n/e\rfloor}_{\substack{r=1\\ (r,n)=1}}(er)^{\phi(n^2)-1}+\sum^{\lfloor n/e\rfloor}_{\substack{r=1\\ (r,n)=1}}(n\phi(n)-1)n(er)^{\phi(n^2)-2}\\
&\equiv -\frac{1}{e}\sum^{\lfloor n/e\rfloor}_{\substack{r=1\\ (r,n)=1}}\frac{1}{r}-\frac{n}{e^2}\sum^{\lfloor n/e\rfloor}_{\substack{r=1\\ (r,n)=1}}\frac{1}{r^2} \pmod {n^2},
\end{align*}
we have
\begin{equation}\label{eq:1.5}
\sum^{\lfloor n/e\rfloor}_{\substack{r=1\\ (r,n)=1}}\frac{1}{r}\equiv -e\sum^{\lfloor n/e\rfloor}_{\substack{r=1\\ (r,n)=1}}\frac{1}{n-er}-\frac{n}{e}\sum^{\lfloor n/e\rfloor}_{\substack{r=1\\ (r,n)=1}}\frac{1}{r^2} \pmod {n^2}.
\end{equation}
Using Theorem \ref{th:1} and Lemma \ref{lem:5}, we can obtain the result.
\end{proof}

\begin{proof}[Proof of Theorem \ref{th:2} and Theorem \ref{th:3}]
For any positive integer $e$, we have
$$\binom{kn-1}{\lfloor n/e\rfloor}=\prod_{r=1}^{\lfloor n/e\rfloor} \frac{kn-r}{r}=\prod_{d\mid n}\prod_{\substack{r=1\\ (r,n)=d}}^{\lfloor n/e\rfloor} \frac{kn-r}{r}=\prod_{d\mid n}T_{n/d}=\prod_{d\mid n}T_d.$$
Here
\begin{equation}\label{eq:2.0}
T_d=\prod_{\substack{r=1\\ (r,d)=1}}^{\lfloor d/e\rfloor} \frac{kd-r}{r}.
\end{equation}
Now it follows by the multiplicative version of M\"obius inversion formula that
\begin{equation}\label{eq:2.1}
T_n=\prod_{d\mid n}\binom{kd-1}{\lfloor d/e\rfloor}^{\mu(n/d)}.
\end{equation}
As for $T_{n}$, we have
\begin{align*}
T_n&=\prod_{\substack{r=1\\ (r,n)=1}}^{\lfloor n/e\rfloor} \frac{kn-r}{r}=(-1)^{\phi_{e}(n)}\prod_{\substack{r=1\\ (r,n)=1}}^{\lfloor n/e\rfloor} \left(1-\frac{kn}{r}\right)\\
&\equiv (-1)^{\phi_{e}(n)}\left\{1-kn\sum_{\substack{r=1\\ (r,n)=1}}^{\lfloor n/e\rfloor} \frac{1}{r}\right.\\
&\quad\quad\quad\quad\quad\quad\  \left.+\frac{k^2 n^2}{2}\left(\left(\sum_{\substack{r=1\\ (r,n)=1}}^{\lfloor n/e\rfloor} \frac{1}{r}\right)^2-\sum_{\substack{r=1\\ (r,n)=1}}^{\lfloor n/e\rfloor}  \frac{1}{r^2}\right)\right\} \pmod{n^3}.
\end{align*}
If $e=2$, then $\lfloor n/2\rfloor=\frac{n-1}{2}$. Note that
$$\sum_{\substack{r=1\\ (r,n)=1}}^{(n-1)/2} \frac{1}{r^2}=\frac{1}{2} \sum_{\substack{r=1\\ (r,n)=1}}^{(n-1)/2} \left\{\frac{1}{r^2}+\frac{1}{(n-r)^2}\right\}=\frac{1}{2}\sum_{\substack{r=1\\ (r,n)=1}}^{n-1} \frac{1}{r^2}\pmod{n}.$$
Now by Lemmas \ref{lem:2} and \ref{lem:4}, if $3\nmid n$ we have
\begin{align*}
T_n&\equiv (-1)^{\phi(n)/2}\left\{1+2knq_2(n)+(2k^2-k)(nq_2(n))^2\right\}\\
&\equiv (-1)^{\phi(n)/2}(1+nq_2(n))^{2k}\\
&\equiv (-1)^{\phi(n)/2}4^{k\phi(n)}\pmod{n^3}.
\end{align*}
If $3\mid n$, we replace the moduli by $n^3/3$. It follows by combining it with \eqref{eq:2.1} that
\begin{equation*}
\prod_{d\mid n}\binom{kd-1}{(d-1)/2}^{\mu(n/d)}\equiv (-1)^{\phi(n)/2}4^{k\phi(n)}
\begin{cases}
(\bmod{\ n^3}) & \text{if }3\nmid n,\\
(\bmod{\ n^3/3}) & \text{if }3\mid n.
\end{cases}
\end{equation*}
If $e=3$, by Theorem \ref{th:1} and Corollary \ref{cor:3}, for $(n,6)=1$, we have
\begin{align*}
T_{n} &\equiv (-1)^{\phi_{3}(n)}\left\{1-kn(-\frac{3}{2}q_{3}(n)+\frac{3}{4}nq_{3}^{2}(n))-\frac{1}{3}kn^{2}A_{3}(n)+\frac{1}{2}k^{2}n^{2}(\frac{9}{4}q_{3}^{2}(n)+A_{3}(n))\right\}\\
&\equiv (-1)^{\phi_{3}(n)}\left\{1+\frac{3}{2}knq_{3}(n)-\frac{3}{4}kn^{2}q_{3}^{2}(n)+\frac{9}{8}k^{2}n^{2}q_{3}^{2}(n)+k(\frac{1}{2}-\frac{1}{3})n^{2}A_{3}(n)\right\}\\
&\equiv (-1)^{\phi_{3}(n)}\left\{\frac{1}{2}((1+nq_{3}(n))^{3k}+1)+k(\frac{1}{2}-\frac{1}{3})n^{2}A_{3}(n)\right\}\\
&\equiv (-1)^{\phi_3(n)}\left\{\frac{1}{2}(27^{k\phi(n)}+1)+k(\frac{1}{2}k-\frac{1}{3})n^{2}A_{3}(n)\right\} \pmod{n^3}
\end{align*}
Similarly, one may deduce \eqref{eq:th:3.2} and \eqref{eq:th:3.3}. This completes the proof of Theorem \ref{th:3}.
\end{proof}

\begin{proof}[Proof of Corollary \ref{cor:5}]
It follows by Theorem \ref{th:2} that we only need to deal with the case $p=3$ and $q\ge 5$. From the proof of Theorem \ref{th:2}, we notice that
$$\binom{3kq-1}{(3q-1)/2}\equiv \left\{4^{2k(q-1)}+\frac{3^2k^2q^2}{4}\sum_{\substack{r=1\\ (r,3q)=1}}^{3q-1}\frac{1}{r^2}\right\}\binom{3k-1}{1}\binom{kq-1}{(q-1)/2} \pmod{3^3q^3}.$$
From Lemma \ref{lem:2}, we have
$$\sum_{\substack{r=1\\ (r,3q)=1}}^{3q-1}\frac{1}{r^2}\equiv 0 \pmod{q}.$$
It also follows by the Fermat's little theorem that
$$\sum_{\substack{r=1\\ (r,3q)=1}}^{3q-1}\frac{1}{r^2}\equiv \sum_{\substack{r=1\\ (r,3q)=1}}^{3q-1} 1=2q-2 \pmod{3}.$$

Now if $q\equiv 1 \pmod{6}$, then
$$\sum_{\substack{r=1\\ (r,3q)=1}}^{3q-1}\frac{1}{r^2}\equiv 0 \pmod{3q}.$$
We therefore obtain \eqref{eq:cor:5}. If $q\equiv 5 \pmod{6}$, then
$$\sum_{\substack{r=1\\ (r,3q)=1}}^{3q-1}\frac{1}{r^2}\equiv 4q \pmod{3q}.$$
To prove Corollary \ref{cor:5}, it suffice to show
\begin{equation}\label{eq:proof.cor}
3^2k^2q^3\binom{3k-1}{1}\binom{kq-1}{(q-1)/2}\equiv 0 \pmod{3^3q^3}.
\end{equation}
For a prime p, let $\mathrm{ord}_p(n):=\max\{i\in\mathbb{N}:p^i\mid n\}$. The Legendre theorem tells
$$\mathrm{ord}_p(n!)=\sum_{i\ge 1}\left\lfloor\frac{n}{p^i}\right\rfloor.$$
We therefore have
\begin{align*}
\mathrm{ord}_3\left(\binom{kq-1}{(q-1)/2}\right)&=\sum_{i\ge 1}\left(\left\lfloor\frac{kq-1}{3^i}\right\rfloor-\left\lfloor\frac{(q-1)/2}{3^i}\right\rfloor-\left\lfloor\frac{((2k-1)q-1)/2}{3^i}\right\rfloor\right)\\
&\ge \left\lfloor\frac{kq-1}{3}\right\rfloor-\left\lfloor\frac{q-1}{6}\right\rfloor-\left\lfloor\frac{(2k-1)q-1}{6}\right\rfloor\\
&= 1+\left\lfloor\frac{5k-1}{3}\right\rfloor-\left\lfloor\frac{5k}{3}\right\rfloor.
\end{align*}
If $3\mid k$, it is obvious that \eqref{eq:proof.cor} holds. If $3\nmid k$, then
$$\mathrm{ord}_3\left(\binom{kq-1}{(q-1)/2}\right)\ge 1,$$
which implies $3\mid \binom{kq-1}{(q-1)/2}$. This leads to \eqref{eq:proof.cor} immediately.
\end{proof}

The proof of Theorem \ref{th:4} is more complicated than the previous two. We note that
\begin{align*}
2^{(n-1)/2}\binom{(kn-1)/2}{(n-1)/2}&=\prod_{r=1}^{(n-1)/2} \frac{kn-(2r-1)}{r}\\
&=\prod_{r=1}^{(n-1)/2} \left(\frac{kn-r}{r}\cdot\frac{kn-(n-r)}{r}\cdot\frac{r}{kn-2r}\right)\\
&=\prod_{d\mid n}\prod_{\substack{r=1\\ (r,n)=d}}^{(n-1)/2} \left(\frac{kn-r}{r}\cdot\frac{kn-(n-r)}{r}\cdot\frac{r}{kn-2r}\right)\\
&=\prod_{d\mid n}S_{n/d}=\prod_{d\mid n}S_d.
\end{align*}
Here
$$S_d=\prod_{\substack{r=1\\ (r,d)=1}}^{(d-1)/2} \left(\frac{kd-r}{r}\cdot\frac{kd-(d-r)}{r}\cdot\frac{r}{kd-2r}\right).$$
Again it follows by the multiplicative version of M\"obius inversion formula that
\begin{align}\label{eq:4.1}
S_n&=\prod_{d\mid n}\left(2^{(d-1)/2}\binom{(kd-1)/2}{(d-1)/2}\right)^{\mu(n/d)}\notag\\
&=2^{\phi(n)/2}\prod_{d\mid n}\binom{(kd-1)/2}{(d-1)/2}^{\mu(n/d)}.
\end{align}

\begin{proof}[Proof of Theorem \ref{th:4}]
Now we assume $3\nmid n$ for convenience. Otherwise, replacing the moduli when needed. In the proof of Theorem \ref{th:2}, we have already obtained
\begin{equation}\label{eq:4.2}
\prod_{\substack{r=1\\ (r,n)=1}}^{(n-1)/2} \frac{kn-r}{r}\equiv (-1)^{\phi(n)/2}4^{k\phi(n)}\pmod{n^3}.
\end{equation}
Next
\begin{align}\label{eq:4.3}
&\prod_{\substack{r=1\\ (r,n)=1}}^{(n-1)/2} \frac{kn-(n-r)}{r}=\prod_{\substack{r=1\\ (r,n)=1}}^{(n-1)/2} \left(1+\frac{(k-1)n}{r}\right)\notag\\
&\quad \equiv 1+(k-1)n\sum_{\substack{r=1\\ (r,n)=1}}^{(n-1)/2} \frac{1}{r}+\frac{(k-1)^2 n^2}{2}\left(\left(\sum_{\substack{r=1\\ (r,n)=1}}^{(n-1)/2} \frac{1}{r}\right)^2-\sum_{\substack{r=1\\ (r,n)=1}}^{(n-1)/2}  \frac{1}{r^2}\right)\notag\\
&\quad \equiv 1-2(k-1)nq_2(n)+(2(k-1)^2+(k-1))(nq_2(n))^2\notag\\
&\quad \equiv (1+nq_2(n))^{-2(k-1)}\notag\\
&\quad \equiv 4^{-(k-1)\phi(n)} \pmod{n^3}.
\end{align}
Then
\begin{align}\label{eq:4.4}
&\prod_{\substack{r=1\\ (r,n)=1}}^{(n-1)/2} \frac{kn-2r}{r}=(-1)^{\phi(n)/2}\prod_{\substack{r=1\\ (r,n)=1}}^{(n-1)/2} \left(2-\frac{kn}{r}\right)\notag\\
&\quad \equiv (-1)^{\phi(n)/2} 2^{\phi(n)/2} \left\{1-\frac{kn}{2}\sum_{\substack{r=1\\ (r,n)=1}}^{(n-1)/2} \frac{1}{r}\right.\notag\\
&\quad\quad\quad\quad\quad\quad\quad\quad\quad\quad \left.+\frac{1}{4}\frac{k^2 n^2}{2}\left(\left(\sum_{\substack{r=1\\ (r,n)=1}}^{(n-1)/2} \frac{1}{r}\right)^2-\sum_{\substack{r=1\\ (r,n)=1}}^{(n-1)/2}  \frac{1}{r^2}\right)\right\}\notag\\
&\quad \equiv (-1)^{\phi(n)/2} 2^{\phi(n)/2} \left\{1+knq_2(n)+\frac{k^2-k}{2}(nq_2(n))^2\right\}\notag\\
&\quad \equiv (-1)^{\phi(n)/2} 2^{\phi(n)/2} (1+nq_2(n))^k\notag\\
&\quad \equiv (-1)^{\phi(n)/2} 2^{\phi(n)/2} 2^{k\phi(n)} \pmod{n^3}.
\end{align}
It follows by \eqref{eq:4.2}, \eqref{eq:4.3} and \eqref{eq:4.4} that
$$S_n=\prod_{\substack{r=1\\ (r,n)=1}}^{(n-1)/2} \left(\frac{kn-r}{r}\cdot\frac{kn-(n-r)}{r}\cdot\frac{r}{kn-2r}\right)\equiv 2^{-(k-3/2)\phi(n)}\pmod{n^3}.$$
Combining it with \eqref{eq:4.1} we therefore have
\begin{equation*}
\prod_{d\mid n}\binom{(kd-1)/2}{(d-1)/2}^{\mu(n/d)}\equiv 2^{-(k-1)\phi(n)}
\begin{cases}
(\bmod{\ n^3}) & \text{if }3\nmid n,\\
(\bmod{\ n^3/3}) & \text{if }3\mid n.
\end{cases}
\end{equation*}
\end{proof}

\subsection*{Acknowledgments}
This work is supported by the National Natural Science Foundation of China (Grant No. 11501052 and Grant No. 11571303).

\bibliographystyle{amsplain}

\end{document}